\documentclass[11pt, reqno]{amsart}
\usepackage{amsmath, amsthm, amscd, amsfonts, amssymb, graphicx, color}
\usepackage[bookmarksnumbered, colorlinks, plainpages]{hyperref}
\usepackage{mathrsfs}
\newtheorem{theorem}{Theorem}[section]
\newtheorem{lemma}[theorem]{Lemma}
\newtheorem{proposition}[theorem]{Proposition}
\newtheorem{corollary}[theorem]{Corollary}
\theoremstyle{definition}
\newtheorem{definition}[theorem]{Definition}

\theoremstyle{remark}
\newtheorem{remark}[theorem]{Remark}
\numberwithin{equation}{section}
\newcommand{\M}{\mathbb{M}}
\newcommand{\C}{\mathbb{C}}

\begin{document}
\setcounter{page}{1}

 \title[Sharpness Loss for Radon--Nikodym Derivatives]{Sharpness Loss for Radon--Nikodym Derivatives of Completely Positive Maps}

\author[ M. Kian]{  Mohsen Kian}

\address{Mohsen Kian: Department of Mathematics, University of Bojnord, P.O. Box 1339, Bojnord
94531, Iran}
\email{kian@ub.ac.ir}

\renewcommand{\subjclassname}{\textup{2020} Mathematics Subject Classification}

\maketitle

\begin{abstract}
The Radon--Nikodym theorem for completely positive maps identifies the order interval below a map with positive contractions in the commutant of its minimal Stinespring representation.  In this paper, we study the behavior of projection-valued Radon--Nikodym derivatives (sharp submaps) under composition.  We prove that sharpness loss is precisely determined  by the multiplicative defect of the induced pullback map. In the finite-dimensional case, we show that this defect is determined by the orthogonal complement of the composite Kraus relation space, denoted $\mathcal{E}_{\Lambda,\Omega}$. A sharp submap remains sharp if and only if $\mathcal{E}_{\Lambda,\Omega}$ reduces the lifted Radon--Nikodym projection. Furthermore, we establish a tensor-factor criterion for the global preservation of sharpness and explicitly quantify the sharpness-loss defect.
\end{abstract}

\maketitle

\section{Introduction}
Radon--Nikodym theory is a basic tool for comparing positive objects.
For measures, it describes absolute continuity by a density.  For
completely positive maps, the analogous density is no longer a function:
it is an operator living in the commutant of a Stinespring
representation.  This operator-theoretic form of the Radon--Nikodym
theorem goes back to foundational work of Arveson and
Belavkin--Staszewski, and it has since become a standard way to describe
the order interval below a completely positive map, see, for
example, \cite{Arveson,BelavkinStaszewski,Paulbook}.

In quantum information theory the same structure appears in the study
of quantum operations and channels.  Finite-dimensional formulations of
Radon--Nikodym derivatives for quantum operations were developed by
Raginsky, where domination of completely positive maps is described in
terms of operator-valued parameters associated with Kraus and
Stinespring representations.  Related Radon--Nikodym and
Lebesgue-decomposition questions have also been studied for
operator-valued completely positive maps, Hilbert-module-valued maps,
and other generalizations; see, among others,
\cite{AG,Dadkhah,GS,GK, MKP,Okayasu,Raginsky}.

A separate but closely related issue is the behavior of
Radon--Nikodym derivatives under composition.  In the classical
measure-theoretic setting, composition and conditioning lead naturally
to chain-rule phenomena.  In the noncommutative setting the situation is
less direct, since the relevant derivatives live in commutants of
Stinespring representations, and these commutants change when completely
positive maps are composed.  Recent work has studied such chain-rule
phenomena for Radon--Nikodym derivatives of completely positive maps;
see \cite{BLJS}.

The feature studied here is sharpness.  A dominated completely positive
map is called sharp, relative to a fixed dominating map, when its
Radon--Nikodym derivative is a projection.  Such projection-valued
derivatives are natural analogues of sharp effects or sharp events.
However, if a completely positive map is composed with another one, a
projection-valued derivative need not remain projection-valued.  Thus
composition can turn a sharp submap into an unsharp one.

The goal of this paper is to describe exactly when this loss of sharpness occurs and to quantify this defect. Given completely positive maps
\[
    \Omega:\mathscr B\to B(\mathcal H),
    \qquad
    \Lambda:\mathscr A\to\mathscr B,
\]
precomposition by \(\Lambda\) induces a unital completely positive map between the corresponding Radon--Nikodym commutants. Instead of focusing purely on the commutant lifting, we analyze the multiplicative defect of this induced map. By realizing it as a compression of a canonical commutant lift, we obtain a reducing-subspace criterion for the preservation of sharpness.

In finite dimensions, this abstract condition reduces to a concrete linear-algebraic test. Specifically, we show that the loss of sharpness is determined by the orthogonal complement of the relation space of the composite Kraus operators $L_jV_i$. This relation space also yields a tensor-factor condition for the global preservation of sharpness.

The paper is organized as follows. Section~2 recalls the required background on completely positive maps, Radon--Nikodym derivatives, sharp submaps, and multiplicative domains. Section~3 constructs the pullback transfer map and its commutant lift. Section~4 derives the compression and defect formulas and proves the sharpness-preservation criterion. Section~5 introduces the finite-dimensional Kraus relation description, the tensor criterion for global sharpness preservation, and concludes with a finite commutative example illustrating the classical probability analogue.

\section{Preliminaries}

Throughout the paper, we assume that $\mathscr{A}$ is a  unital $C^*$-algebra and $B(\mathcal H)$ is the $C^*$-algebra of all  bounded operators on a Hilbert space $\mathcal H$.

  A linear map  $\Phi:\mathscr A\to B(\mathcal H)$ is called completely positive if, for every $n\geq 1$, the amplified map
\[
    \Phi^{(n)}:\M_n(\mathscr A)\to \M_n(B(\mathcal H)),
    \qquad
    \Phi^{(n)}([A_{ij}])=[\Phi(A_{ij})],
\]
is positive.  We   write  $ \Phi\leq \Psi$,
when $\Psi-\Phi$ is completely positive.

Let $\Theta:\mathscr A\to\mathscr B$ be a unital completely positive map between unital $C^*$-algebras. Its multiplicative domain is
\[
    \operatorname{MD}(\Theta)
    =
    \{X\in\mathscr A:
    \Theta(X^*X)=\Theta(X)^*\Theta(X)
    \text{ and }
    \Theta(XX^*)=\Theta(X)\Theta(X)^*\}.
\]

The celebrated  Stinespring's theorem asserts that if    $\Omega:\mathscr A\to B(\mathcal H)$ is  completely positive, then  there exist a Hilbert space $\mathcal K$, a
unital $*$-representation $\pi:\mathscr A\to B(\mathcal K)$,  and a bounded operator
$ V:\mathcal H\to\mathcal K$ such that
\[
    \Omega(A)=V^*\pi(A)V,
    \qquad A\in\mathscr A.
\]
The triple $(\pi,\mathcal K,V)$ is called a Stinespring representation of $\Omega$. This representation is called  minimal, when
\[
    \mathcal K
    =
    \overline{\operatorname{span}}
    \{\pi(A)V\xi:A\in\mathscr A,\ \xi\in\mathcal H\}.
\]
A minimal Stinespring representation is unique up to unitary equivalence.

The following Radon--Nikodym theorem (see \cite{Arveson,Paulbook}) for completely positive maps is the
basic order-theoretic tool of the paper.

\textbf{Theorem.}  Let $\Omega:\mathscr A\to B(\mathcal H)$ be completely positive, and let $(\pi,\mathcal K,V)$
be a minimal Stinespring representation of $\Omega$. For every completely positive map $\Phi\leq \Omega$, there exists a unique positive contraction $D_\Phi\in  \pi(\mathscr A)'$, such that
\[
    \Phi(A)=V^*\pi(A)D_\Phi V,
    \qquad A\in\mathscr A.
\]
Conversely, every  positive contraction $D_\Phi\in  \pi(\mathscr A)'$ defines a
completely positive map
\[
    \Omega_D:\mathscr A\to B(\mathcal H),\qquad \Omega_D(A)=V^*\pi(A)DV,
    \qquad A\in\mathscr A
\]
and this map satisfies $\Omega_D\leq\Omega$.

Thus the correspondence
\[
    \Phi\longleftrightarrow D_\Phi
\]
is an affine order isomorphism.

The operator $D_\Phi$ is called the Radon--Nikodym derivative of $\Phi$
with respect to $\Omega$, and we write
\[
    D_\Phi=\frac{\mathrm{d}\Phi}{\mathrm{d}\Omega}.
\]
A completely positive map $\Phi\leq \Omega$ is said to be  sharp relative to $\Omega$ if $D_\Phi$ is a projection.

If  $  \Omega:\M_m\to \M_n$  is a   completely positive map between matrix algebras, then there is a well-known Kraus representation \cite{Bhatia} of $\Omega$  given by
\[
    \Omega(B)=\sum_{i=1}^r V_i^* B V_i,
    \qquad B\in \M_m,
\]
for suitable matrices $V_i$. This representation is called minimal if the
Kraus operators $V_1,\ldots,V_r$ are linearly independent.

  \section{Radon--Nikodym pullback under precomposition}
In this section, we construct the pullback transfer map using standard Stinespring representations and commutant lifting. The  underlying non-commutative techniques utilized here are classical. But, we establishes the notation and basic properties required for the defect analysis in next sections.

Let $\Omega:\mathscr B\to B(\mathcal H)$ and  $\Lambda:\mathscr A\to \mathscr B$ be completely positive maps. Fix a minimal Stinespring representation
\[
    \Omega(B)=V^*\pi(B)V,
    \qquad B\in\mathscr B,
\]
on a Hilbert space $\mathcal K_\Omega$. For every positive operator  $ D\in \pi(\mathscr B)'_+$,
define a completely positive map $\Omega_D:\mathscr B\to B(\mathcal H)$ by
\[
    \Omega_D(B)=V^*\pi(B)DV,
    \qquad B\in\mathscr B.
\]
Since  $0\leq D\leq \|D\|I$, we have
\[
    0\leq \Omega_D\leq \|D\|\Omega.
\]
Hence
\begin{align}\label{domi-1}
    0\leq \Omega_D\circ\Lambda
    \leq
    \|D\|(\Omega\circ\Lambda).
\end{align}

\begin{definition}[Pullback transfer map]
Let $ \Omega:\mathscr B\to B(\mathcal H)$ and $ \Lambda:\mathscr A\to \mathscr B$ be completely positive maps. We define
\[
    \mathcal T_\Lambda^\Omega(D)
    =
    \frac{\mathrm{d}(\Omega_D\circ\Lambda)}
         {\mathrm{d}(\Omega\circ\Lambda)},\qquad D\in \pi_\Omega(\mathscr B)'_+,
\]
where the Radon--Nikodym derivative is computed with respect to a
minimal Stinespring representation of $\Omega\circ\Lambda$.
The map
\[
    D\mapsto \mathcal T_\Lambda^\Omega(D)
\]
will be called the pullback transfer map associated with the pair
$(\Omega,\Lambda)$.
\end{definition}
Note that, by definition,
\begin{align}\label{perb-def}
    \Omega_D\circ\Lambda
    =
    (\Omega\circ\Lambda)_{\mathcal T_\Lambda^\Omega(D)}\qquad (D\in \pi_\Omega(\mathscr B)'_+).
\end{align}
Equivalently, if $ 0\leq \Phi\leq \Omega$ and  $D_\Phi=\frac{\mathrm{d}\Phi}{\mathrm{d}\Omega}$, then
\[
    \mathcal T_\Lambda^\Omega(D_\Phi)
    =
    \frac{\mathrm{d}(\Phi\circ\Lambda)}
         {\mathrm{d}(\Omega\circ\Lambda)}.
\]

The following proposition records the basic functorial properties of the
Radon--Nikodym pullback.  These properties are direct consequences of
the Radon--Nikodym theorem for completely positive maps, together with
its matrix amplifications; we include the proof to fix notation and to
make clear how the pullback acts on derivatives.
\begin{proposition}
\label{prop-basic-pullback}
The assignment
\[
    D\in \pi_\Omega(\mathscr B)'_+
    \longmapsto
    \mathcal T_\Lambda^\Omega(D)
    =
    \frac{\mathrm d(\Omega_D\circ\Lambda)}
         {\mathrm d(\Omega\circ\Lambda)}
\]
is additive and positively homogeneous.  Hence it extends uniquely to a
positive linear map
\[
    \mathcal T_\Lambda^\Omega:
    \pi_\Omega(\mathscr B)'
    \longrightarrow
    \pi_{\Omega\circ\Lambda}(\mathscr A)'.
\]
This extension is unital and completely positive.  Moreover, for every
\(0\leq \Phi\leq\Omega\),
\[
    \mathcal T_\Lambda^\Omega
    \left(
        \frac{\mathrm d\Phi}{\mathrm d\Omega}
    \right)
    =
    \frac{\mathrm d(\Phi\circ\Lambda)}
         {\mathrm d(\Omega\circ\Lambda)}.
\]
If \(\Omega\) and \(\Lambda\) are normal maps between von Neumann
algebras and the Stinespring representations are taken in the normal
category, then \(\mathcal T_\Lambda^\Omega\) is normal.
\end{proposition}

\begin{proof}
We first prove additivity and positive homogeneity on the positive cone.
It follows from  definition of $\Omega_D$ that
\[
    \Omega_{\alpha D+\beta E}
    =
    \alpha\Omega_D+\beta\Omega_E
\]
 for all $ D,E\in\pi(\mathscr B)'_+$ and  $\alpha,\beta\geq0$.
Hence
\[
    \Omega_{\alpha D+\beta E}\circ\Lambda
    =
    \alpha(\Omega_D\circ\Lambda)
    +
    \beta(\Omega_E\circ\Lambda).
\]
The  uniqueness of Radon--Nikodym derivatives then gives
\[
    \mathcal T_\Lambda^\Omega(\alpha D+\beta E)
    =
    \alpha\mathcal T_\Lambda^\Omega(D)
    +
    \beta\mathcal T_\Lambda^\Omega(E).
\]
Since the positive cone linearly generates $\pi(\mathscr B)'$, this additive homogeneous map extends uniquely to   a linear map on
$\pi(\mathscr B)'$. This extension is positive by construction. Moreover, as $ \Omega_I=\Omega$, we have
\[
    \mathcal T_\Lambda^\Omega(I)
    =
    \frac{\mathrm{d}(\Omega\circ\Lambda)}
         {\mathrm{d}(\Omega\circ\Lambda)}
    =
    I.
\]
Thus $\mathcal T_\Lambda^\Omega$ is unital.

2. Complete positivity: We now prove complete positivity. Let
\[
    [D_{ij}]\in \M_n(\pi(\mathscr B)')
\]
be positive. Let
\[
    \widetilde V:\mathcal H^n\to\mathcal K_\Omega^n
\]
be the diagonal amplification of $V$, and let
\[
    \pi^{(n)}(B)
    =
    \operatorname{diag}(\pi(B),\ldots,\pi(B)),\qquad (B\in\mathscr{B}).
\]
Now consider the map $\Psi:\mathscr A\to B(\mathcal H^n)$ defined by
\[
    \Psi(A)
    =
    \widetilde V^*
    \pi^{(n)}(\Lambda(A))
    [D_{ij}]
    \widetilde V.
\]
This is a completely positive map, since  $[D_{ij}]\in \pi^{(n)}(\mathscr B)'$ and so   we can  write
\[
    \Psi(A)
    =
    ([D_{ij}]^{1/2}\widetilde V)^*
    \pi^{(n)}(\Lambda(A))
    ([D_{ij}]^{1/2}\widetilde V).
\]
Also, it follows from
\[
    0\leq [D_{ij}]
    \leq
    \|[D_{ij}]\|I_{\mathcal K_\Omega^n}.
\]
that  for every \(A\in\mathscr A_+\),
\begin{align*}
    0\leq \Psi(A)
   & \leq
    \|[D_{ij}]\|
    \widetilde V^*\pi^{(n)}(\Lambda(A))\widetilde V\\
    & = \|[D_{ij}]\| \operatorname{diag}
    \big(
        (\Omega\circ\Lambda)(A),\ldots,
        (\Omega\circ\Lambda)(A)
    \big)\\
    &=  \|[D_{ij}]\|(\Omega\circ\Lambda)^{(n)},
\end{align*}
 where \((\Omega\circ\Lambda)^{(n)}\) denotes the diagonal amplification of
\(\Omega\circ\Lambda\). Let
\begin{align}\label{stein-ol}
    \Omega\circ\Lambda(A)=W^*\rho(A)W
\end{align}
be a minimal Stinespring representation of $\Omega\circ\Lambda$. Then, applying \eqref{perb-def} gives
 \[
    W^*\rho(A)\mathcal T_\Lambda^\Omega(D)W
    =
    V^*\pi(\Lambda(A))DV.
\]
 for every \(D\in\pi(\mathscr B)'\) and every
\(A\in\mathscr A\).
Applying this to each entry \(D_{ij}\), we obtain
\[
    V^*\pi(\Lambda(A))D_{ij}V
    =
    W^*\rho(A)\mathcal T_\Lambda^\Omega(D_{ij})W
\]
 for every \(A\in\mathscr A\). Consequently, we have
\[
    \Psi(A)
    =
    \big[
        W^*\rho(A)\mathcal T_\Lambda^\Omega(D_{ij})W
    \big]_{i,j=1}^n.
\]
Equivalently, if \(\widetilde W:\mathcal H^n\to\mathcal K^n\) denotes
the diagonal amplification of \(W\), then
\[
    \Psi(A)
    =
    \widetilde W^*\rho^{(n)}(A)
    \big[
        \mathcal T_\Lambda^\Omega(D_{ij})
    \big]_{i,j=1}^n
    \widetilde W.
\]
Thus, by uniqueness in the Radon--Nikodym theorem, the derivative of
\(\Psi\) with respect to \((\Omega\circ\Lambda)^{(n)}\) is
\[
    \big[
        \mathcal T_\Lambda^\Omega(D_{ij})
    \big]_{i,j=1}^n.
\]
Since $\Psi$ is completely positive, this derivative is positive. Hence
\[
    \big[
        \mathcal T_\Lambda^\Omega(D_{ij})
    \big]_{i,j=1}^n
    \geq0.
\]
Therefore $\mathcal T_\Lambda^\Omega$ is completely positive.

3. Normality:  Let
\[
    0\leq D_\alpha\uparrow D
\]
be an increasing bounded net in $\pi(\mathscr B)'$.

For $A\in\mathscr A_+$, the operator $\pi(\Lambda(A))$ is positive and
commutes with each $D_\alpha$. Therefore
\[
    \pi(\Lambda(A))D_\alpha
    =
    \pi(\Lambda(A))^{1/2}D_\alpha\pi(\Lambda(A))^{1/2}
    \uparrow
    \pi(\Lambda(A))^{1/2}D\pi(\Lambda(A))^{1/2}
    =
    \pi(\Lambda(A))D
\]
in the weak operator topology. Hence
\begin{align}\label{net-omega}
    (\Omega_{D_\alpha}\circ\Lambda)(A)
    \uparrow
    (\Omega_D\circ\Lambda)(A).
\end{align}
Set
\[
    E_\alpha=\mathcal T_\Lambda^\Omega(D_\alpha)
    \qquad\text{and}\qquad
    E=\mathcal T_\Lambda^\Omega(D).
\]
Since $\mathcal T_\Lambda^\Omega$ is positive, we have $0\leq E_\alpha\leq E$ and so  $(E_\alpha)$ is an increasing bounded net in the von Neumann algebra
$\rho(\mathscr A)'$. Put $ F=\sup_\alpha E_\alpha$.  Note that $F$ and the $E_\alpha$ commute with $\rho(A)$ for  every $A\in\mathscr A_+$ so
\begin{align*}
    W^*\rho(A)FW
    &=
    \sup_\alpha W^*\rho(A)E_\alpha W\\
    &=\sup_\alpha W^*\rho(A)\mathcal T_\Lambda^\Omega(D_\alpha)  W\\
    &= \sup_\alpha
    (\Omega_{D_\alpha}\circ\Lambda)(A)\qquad (\text{by \eqref{perb-def} and \eqref{stein-ol}})\\
    &=
    (\Omega_D\circ\Lambda)(A)  \qquad (\text{by \eqref{net-omega}})\\
    &=
    W^*\rho(A)EW.
\end{align*}
By uniqueness of the Radon--Nikodym derivative in the minimal
Stinespring representation of $\Omega\circ\Lambda$, we get $F=E$, meaning that
\[
    \mathcal T_\Lambda^\Omega(D_\alpha)
    \uparrow
    \mathcal T_\Lambda^\Omega(D).
\]
Thus $\mathcal T_\Lambda^\Omega$ is normal.

4. If $0\leq \Phi\leq\Omega$ and $ D_\Phi=\frac{\mathrm{d}\Phi}{\mathrm{d}\Omega}$,
then $    \Phi=\Omega_{D_\Phi}$. Consequently,
\[
    \Phi\circ\Lambda
    =
    \Omega_{D_\Phi}\circ\Lambda,
\]
and this gives the desired identity
\[
    \mathcal T_\Lambda^\Omega(D_\Phi)
    =
    \frac{\mathrm{d}(\Phi\circ\Lambda)}
         {\mathrm{d}(\Omega\circ\Lambda)}.
\]
\end{proof}


  \subsection{The compression formula}

We next relate the pullback transfer map to a standard Stinespring
compression.  Start with a minimal Stinespring representation
\[
    \Omega(B)=V^*\pi(B)V .
\]
Applying Stinespring's theorem to the completely positive map
\(\pi\circ\Lambda\) gives a larger representation space.  Inside this
space one finds the minimal Stinespring space for the composite map
\(\Omega\circ\Lambda\).  The transfer map
\(\mathcal T_\Lambda^\Omega\) is obtained by compressing the natural
commutant lift to this smaller space.

  Let
\[
    \Omega(B)=V^*\pi(B)V,
    \qquad B\in\mathscr B,
\]
be the fixed minimal Stinespring representation of \(\Omega\) on
\(\mathcal K_\Omega\).  Consider the completely positive map
\[
    \pi\circ\Lambda:\mathscr A\to B(\mathcal K_\Omega).
\]
Let
\[
    \pi(\Lambda(A))=S^*\varphi(A)S,
    \qquad A\in\mathscr A,
\]
be a minimal Stinespring representation of \(\pi\circ\Lambda\) on a
Hilbert space \(\mathcal L\).  Then
\[
    \Omega\circ\Lambda(A)
    =
    (SV)^*\varphi(A)(SV),
    \qquad A\in\mathscr A.
\]

Set
\[
    \mathcal K_{\Omega,\Lambda}
    =
    \overline{\operatorname{span}}
    \{\varphi(A)SV\xi:A\in\mathscr A,\ \xi\in\mathcal H\}
    \subseteq \mathcal L .
\]
This subspace reduces \(\varphi(\mathscr A)\).  Indeed, for
\(C,A\in\mathscr A\) and \(\xi\in\mathcal H\),
\[
    \varphi(C)\varphi(A)SV\xi
    =
    \varphi(CA)SV\xi,
    \qquad
    \varphi(C)^*\varphi(A)SV\xi
    =
    \varphi(C^*A)SV\xi.
\]
Thus the restriction
\[
    \rho(A)=\varphi(A)|_{\mathcal K_{\Omega,\Lambda}}
\]
defines a \(*\)-representation.  With \(W=SV\), the triple
\[
    (\rho,W,\mathcal K_{\Omega,\Lambda})
\]
is a minimal Stinespring representation of \(\Omega\circ\Lambda\).

The following  is a standard application of Stinespring dilation theory; we include the proof to fix the notation used in our compression formulas
  \begin{lemma}\label{lift-com}
For every \(D\in\pi(\mathscr B)'\), there is a unique operator
\(\Gamma_\Lambda^\Omega(D)\in\varphi(\mathscr A)'\) such that
\[
    \Gamma_\Lambda^\Omega(D)\varphi(A)S\eta
    =
    \varphi(A)SD\eta,
    \qquad
    A\in\mathscr A,\ \eta\in\mathcal K_\Omega .
\]
The map
\[
    \Gamma_\Lambda^\Omega:
    \pi(\mathscr B)'
    \longrightarrow
    \varphi(\mathscr A)'
\]
is a unital \(*\)-homomorphism.
\end{lemma}

\begin{proof}
Since $(\varphi,S,\mathcal L)$ is a minimal Stinespring representation of $\pi\circ \Lambda$, the subspace \[
    \operatorname{span}
    \{\varphi(A)S\eta:A\in\mathscr A,\ \eta\in\mathcal K_\Omega\}
\]
is dense in $\mathcal L$.  We define   the operator
\[
    \Gamma_\Lambda^\Omega(D)
    \left(
        \sum_{k=1}^n\varphi(A_k)S\eta_k
    \right)
    =
    \sum_{k=1}^n\varphi(A_k)SD\eta_k
\]
on this subspace. We show that this is well-defined and bounded. If $
    A_1,\ldots,A_n\in\mathscr A$,
then $[A_j^*A_k]_{j,k=1}^n$ is positive. Since $\pi\circ\Lambda$ is completely positive,  we have
\[
   0\leq  P=
    [\pi(\Lambda(A_j^*A_k))]_{j,k=1}^n
    \in B(\mathcal K_\Omega^n).
\]
Since $D\in \pi(\mathscr B)'$, it commutes   with every $\pi(\Lambda(A_j^*A_k))$. Consequently, the operator $D^{(n)}=\operatorname{diag}(D,\ldots,D)$ commutes with $P$. Therefore
\begin{align}\label{nrm-1}
    D^{(n)*}PD^{(n)}
    =
    P^{1/2}D^{(n)*}D^{(n)}P^{1/2}
    \leq
    \|D\|^2P.
\end{align}
Writing \(\eta=(\eta_1,\ldots,\eta_n)\), we have
\[
\begin{aligned}
\left\|
    \sum_{k=1}^n\varphi(A_k)SD\eta_k
\right\|^2
&=
\sum_{j,k=1}^n
\left\langle
    \varphi(A_j)SD\eta_j,
    \varphi(A_k)SD\eta_k
\right\rangle  \\
&=\sum_{j,k=1}^n
\left\langle
     D\eta_j,
    S^*\varphi(A_j^*A_k)SD\eta_k
\right\rangle \\
&=
\sum_{j,k=1}^n
\left\langle
    D\eta_j,
    \pi(\Lambda(A_j^*A_k))D\eta_k
\right\rangle \\
&\qquad\qquad\qquad (\text{since $(\varphi,S,\mathcal{L})$ is a Stinespring representation of $\pi\circ\Lambda$}) \\
&=
\left\langle
    D^{(n)}\eta,
    P D^{(n)}\eta
\right\rangle \\
    &=
    \left\langle
        \eta,
        D^{(n)*}PD^{(n)}\eta
    \right\rangle                                      \\
    &\leq
    \|D\|^2
    \left\langle
        \eta,
        P\eta
    \right\rangle\qquad\qquad (\text{by \eqref{nrm-1}})                                      \\
    &=
    \|D\|^2
    \left\|
        \sum_{k=1}^n\varphi(A_k)S\eta_k
    \right\|^2.
\end{aligned}
\]
Hence the densely defined operator is bounded with norm at most
$\|D\|$. In particular, it is well-defined and extends uniquely to a
bounded operator on $\mathcal L$. Moreover, by defining formula of  $\Gamma_\Lambda^\Omega(D)$, we have
\[
    \Gamma_\Lambda^\Omega(D)\sigma(A)
    =
    \sigma(A)\Gamma_\Lambda^\Omega(D),
    \qquad A\in\mathscr A.
\]
This guarantees $\Gamma_\Lambda^\Omega(D)\in\sigma(\mathscr A)'$.  In addition, we have clearly  $    \Gamma_\Lambda^\Omega(I)=I$ and
\[
    \Gamma_\Lambda^\Omega(D_1D_2)
    =
    \Gamma_\Lambda^\Omega(D_1)\Gamma_\Lambda^\Omega(D_2).
\]
To complete the proof, we need to show that
\[
    \Gamma_\Lambda^\Omega(D^*)=\Gamma_\Lambda^\Omega(D)^*\qquad D\in\pi(\mathscr B)'.
\]
Indeed, for $A,B\in\mathscr A$ and
$\eta,\zeta\in\mathcal K_\Omega$, we have
\[
\begin{aligned}
    \left\langle
        \Gamma_\Lambda^\Omega(D)\varphi(A)S\eta,
        \varphi(B)S\zeta
    \right\rangle       & =
    \left\langle
        \varphi(A)SD\eta,
        \varphi(B)S\zeta
    \right\rangle                                      \\
    &\quad =
    \left\langle
        D\eta,
        S^*\varphi(A^*B)S\zeta
    \right\rangle                                      \\
    &\quad =
    \left\langle
        D\eta,
        \pi(\Lambda(A^*B))\zeta
    \right\rangle   \qquad(\text {since $S^*\varphi(\cdot)S=\pi\circ\Lambda(\cdot)$})                                     \\
    &\quad =
    \left\langle
        \eta,
        \pi(\Lambda(A^*B))D^*\zeta
    \right\rangle   \quad(\text {since $D$ commutes with each $\pi(\Lambda(X))$})                                   \\
    &\quad =
    \left\langle
        \varphi(A)S\eta,
        \varphi(B)SD^*\zeta
    \right\rangle                                      \\
    &\quad =
    \left\langle
        \varphi(A)S\eta,
        \Gamma_\Lambda^\Omega(D^*)\varphi(B)S\zeta
    \right\rangle .
\end{aligned}
\]
Since vectors of the form $\varphi(A)S\eta$ are dense in $\mathcal L$,
we obtain
\[
    \Gamma_\Lambda^\Omega(D)^*
    =
    \Gamma_\Lambda^\Omega(D^*).
\]
Thus $\Gamma_\Lambda^\Omega$ is a unital $*$-homomorphism.
\end{proof}


The following theorem expresses the pullback transfer map as a compression of the $*$-homomorphic lift $\Gamma_\Lambda^\Omega$.

\begin{theorem}
Let \(\Omega\), \(\Lambda\), and
\(\Gamma_\Lambda^\Omega\) be as in Lemma~\ref{lift-com}. Then the
pullback transfer map is obtained by compressing the
\(*\)-homomorphic lift \(\Gamma_\Lambda^\Omega\) to the minimal
Stinespring space of \(\Omega\circ\Lambda\). More precisely, for every
\(D\in\pi(\mathscr B)'\),
\begin{align}\label{compres-formul}
    \mathcal T_\Lambda^\Omega(D)
    =
    Q_{\Omega,\Lambda}
    \Gamma_\Lambda^\Omega(D)
    Q_{\Omega,\Lambda}
    \big|_{\mathcal K_{\Omega,\Lambda}},
\end{align}
where \(Q_{\Omega,\Lambda}\) denotes the orthogonal projection from
\(\mathcal L\) onto \(\mathcal K_{\Omega,\Lambda}\).
\end{theorem}
\begin{proof}
Since both sides of \eqref{compres-formul} are linear in $D$, we only need  to prove this formula for $D\in\pi(\mathscr B)'_+$.
Put
\[
    E=
    Q_{\Omega,\Lambda}
    \Gamma_\Lambda^\Omega(D)
    Q_{\Omega,\Lambda}
    \big|_{\mathcal K_{\Omega,\Lambda}} .
\]
Since $Q_{\Omega,\Lambda}$ and  $\Gamma_\Lambda^\Omega(D)$ both commute with $\varphi(A)$, their compression $E$ commutes with the restricted representation $\rho(A)=\varphi(A)|_{\mathcal K_{\Omega,\Lambda}}$.
Moreover, if $D\geq0$, then $\Gamma_\Lambda^\Omega(D)\geq0$, and hence $E\geq0$.
For every $A\in\mathscr A$ we have
\begin{align}\label{rep-w}
    W^*\rho(A)EW
    &=
    (SV)^*
    \varphi(A)
    Q_{\Omega,\Lambda}
    \Gamma_\Lambda^\Omega(D)
    Q_{\Omega,\Lambda}
    SV .
\end{align}
It follows from definition of $\mathcal K_{\Omega,\Lambda}$ that  $ SV\mathcal H\subseteq \mathcal K_{\Omega,\Lambda}$, and so
\begin{align}\label{q-commu}
 Q_{\Omega,\Lambda}SV=SV\qquad\text{and}\qquad (SV)^*Q_{\Omega,\Lambda}=(SV)^*.
\end{align}
Since \(Q_{\Omega,\Lambda}\) commutes with \(\varphi(A)\), this implies that
 \begin{align}\label{sv-3}
    (SV)^*\varphi(A)Q_{\Omega,\Lambda}
    =
    (SV)^*Q_{\Omega,\Lambda}\varphi(A)
    =
    (SV)^*\varphi(A).
\end{align}
Therefore
\begin{align*}
    W^*\rho(A)EW
    &=
    (SV)^*\varphi(A)\Gamma_\Lambda^\Omega(D)SV  \qquad(\text{by \eqref{rep-w},\eqref{q-commu} and \eqref{sv-3}})      \\
    &=
    (SV)^*\varphi(A)SDV \qquad(\text{by definition of $\Gamma_\Lambda^\Omega$ })  \\
    &=
    V^*S^*\varphi(A)SDV  \\
    &=
    V^*\pi(\Lambda(A))DV \qquad (\text{since $(\varphi,S,\mathcal{L})$ is a Stinespring representation of $\pi\circ\Lambda$})  \\
    &=
    (\Omega_D\circ\Lambda)(A).
\end{align*}
Thus $E$ is the Radon--Nikodym derivative of
$\Omega_D\circ\Lambda$ with respect to $\Omega\circ\Lambda$. Hence
\[
    E=\mathcal T_\Lambda^\Omega(D),
\]
as required.
\end{proof}


\subsection{The multiplicative defect}

We keep the notation of the previous subsection and write
\[
    \mathcal T=\mathcal T_\Lambda^\Omega,
    \qquad
    \Gamma=\Gamma_\Lambda^\Omega,
    \qquad
    Q=Q_{\Omega,\Lambda}.
\]
Thus
\[
    \mathcal T(D)=Q\Gamma(D)Q|_{Q\mathcal L},
    \qquad D\in\pi(\mathscr B)'.
\]

To measure the deviation of this compression from multiplicativity, we apply standard multiplicative-domain theory. For a unital completely positive map, equality in the Schwarz inequality determines the largest $C^*$-subalgebra on which the map is multiplicative; see, for example, \cite{Choi,Paulbook,Rahaman}.

By making the loss of multiplicativity explicit, the standard defect identities allow us to translate the order-theoretic property of sharpness into a geometric subspace reduction condition, which is  the foundation for the finite-dimensional Kraus relations in next section.

The next lemma details the compression defect identity for this particular pullback map.

\begin{lemma}\label{def-lem}
With notations as above,
\begin{align}\label{defect-1}
    \mathcal T(D^*D)-\mathcal T(D)^*\mathcal T(D)
    =
    Q\Gamma(D)^*(I-Q)\Gamma(D)Q
    \big|_{Q\mathcal L}
\end{align}
for every  $D\in\pi(\mathscr B)'$.
Similarly,
\begin{align}\label{defect-2}
    \mathcal T(DD^*)-\mathcal T(D)\mathcal T(D)^*
    =
    Q\Gamma(D)(I-Q)\Gamma(D)^*Q
    \big|_{Q\mathcal L}.
\end{align}
\end{lemma}

\begin{proof}
Using the   formula \eqref{compres-formul} and the fact that $\Gamma$ is a
$*$-homomorphism, we have
\begin{align}\label{qt-1}
    \mathcal T(D^*D)
    =
    Q\Gamma(D^*D)Q|_{Q\mathcal L}  =
    Q\Gamma(D)^*\Gamma(D)Q|_{Q\mathcal L}.
\end{align}
On the other hand,
\begin{align}\label{qt-2}
    \mathcal T(D)^*\mathcal T(D)
    =
    Q\Gamma(D)^*Q\Gamma(D)Q|_{Q\mathcal L}.
\end{align}
Subtracting \eqref{qt-2} from \eqref{qt-1} proves \eqref{defect-1}. The formula  \eqref{defect-2}
  follows by applying the  \eqref{defect-1}  to $D^*$.
\end{proof}

The preceding defect identities allow us to explicitly characterize the multiplicative domain of the pullback map.
\begin{theorem}\label{th-loss}
Let  $D\in\pi(\mathscr B)'$. Then  $ D\in\operatorname{MD}(\mathcal T_\Lambda^\Omega)$ if and only if
    $$\Gamma_\Lambda^\Omega(D)Q_{\Omega,\Lambda}
    =
    Q_{\Omega,\Lambda}\Gamma_\Lambda^\Omega(D).$$
Consequently,
\[
   \mathrm{MD}(\mathcal T_\Lambda^\Omega)
    =
    \{D\in\pi(\mathscr B)':
    \Gamma_\Lambda^\Omega(D)Q_{\Omega,\Lambda}
    =
    Q_{\Omega,\Lambda}\Gamma_\Lambda^\Omega(D)\}.
\]
\end{theorem}

\begin{proof}
By definition, $D\in\operatorname{MD}(\mathcal T)$ if and only if
\[
    \mathcal T(D^*D)=\mathcal T(D)^*\mathcal T(D)\qquad \text{and}\qquad
    \mathcal T(DD^*)=\mathcal T(D)\mathcal T(D)^*.
\]
By the Lemma~\ref{def-lem}, these are    equivalent to
\begin{align}\label{qd-1}
    Q\Gamma(D)^*(I-Q)\Gamma(D)Q=0 \qquad \text{and}\qquad (I-Q)\Gamma(D)^*Q=0.
\end{align}
Since
\[
    Q\Gamma(D)^*(I-Q)\Gamma(D)Q
    =
    \big((I-Q)\Gamma(D)Q\big)^*
    \big((I-Q)\Gamma(D)Q\big),
\]
\eqref{qd-1} are further  equivalent to
\begin{align}\label{qd-1}
    (I-Q)\Gamma(D)Q=0\qquad\text{and}\qquad   Q\Gamma(D)(I-Q)=0.
\end{align}
These two conditions say exactly that
\[
    \Gamma(D)Q=Q\Gamma(D).
\]
This proves the claim.
\end{proof}
Therefore, the multiplicative domain of the pullback map has a natural geometric characterization: it is the largest $C^*$-subalgebra of $\pi(\mathscr B)'$ whose lifted image reduces the minimal Stinespring space $\mathcal K_{\Omega,\Lambda}$. In operator-theoretic terms, it identifies the part of the Radon--Nikodym commutant on which precomposition by $\Lambda$ preserves multiplicativity.

Since a sharp submap corresponds to a projection-valued derivative, Theorem 4.2 directly implies the following structural criterion for sharpness preservation
\begin{corollary}
Let $ P\in\pi(\mathscr B)'$ be a projection. Then $\mathcal T_\Lambda^\Omega(P)$ is a projection if and only if
\[
    \Gamma_\Lambda^\Omega(P)Q_{\Omega,\Lambda}
    =
    Q_{\Omega,\Lambda}\Gamma_\Lambda^\Omega(P).
\]
Equivalently, a sharp submap
\[
    \Omega_P\leq\Omega
\]
remains sharp after pullback by $\Lambda$ if and only if the lifted
projection $\Gamma_\Lambda^\Omega(P)$ reduces
$\mathcal K_{\Omega,\Lambda}$.
\end{corollary}

\begin{proof}
Since $P=P^*=P^2$, we have $P\in\operatorname{MD}(\mathcal T)$ if and only if $ \mathcal T(P)^2=\mathcal T(P)$. Because $\mathcal T(P)$ is a positive contraction, this is equivalent to
$\mathcal T(P)$ being a projection. The result now follows from
Theorem~\ref{th-loss}.
\end{proof}

\section{The finite-dimensional Kraus relation formula}
We now specialize the pullback construction to matrix algebras.  In
finite dimensions, Radon--Nikodym derivatives of completely positive
maps can be described explicitly using Kraus and Stinespring data; see,
for example, \cite{Raginsky,Paulbook}.  The point of this section is to
write the particular pullback map \(\mathcal T_\Lambda^\Omega\) in terms
of the relation space of the composite Kraus family \(L_jV_i\).  The
resulting formula turns the abstract sharpness criterion into a
finite-dimensional linear-algebraic test.

Let $\Omega:\M_m\to \M_n$ be completely positive, and fix a minimal Kraus representation
\[
    \Omega(B)=\sum_{i=1}^r V_i^*BV_i,
    \qquad B\in \M_m.
\]
Thus \(V_1,\ldots,V_r\) are linearly independent matrices from
\(\mathbb C^n\) to \(\mathbb C^m\).  Moreover, under the standard finite-dimensional Stinespring identification associated
with the minimal Kraus representation of \(\Omega\), we have
\[
    \pi_\Omega(\M_m)' \cong \M_r.
\]
Let $\Lambda:\M_\ell\to \M_m$
be completely positive, and choose a Kraus representation
\begin{align}\label{kraus-lambda}
    \Lambda(A)=\sum_{j=1}^s L_j^*AL_j,
    \qquad A\in \M_\ell.
\end{align}
This representation of \(\Lambda\) need not be minimal.  The possible
linear dependence introduced by this choice will be recorded by the
relation space below. Then
\[
    \Omega\circ\Lambda(A)
    =
    \sum_{j=1}^s\sum_{i=1}^r
    (L_j V_i)^*A(L_j V_i).
\]
We write
\[
    W_{j i}=L_j V_i,
    \qquad
    1\leq j\leq s,\ 1\leq i\leq r,
\]
for the   Kraus family of $\Omega\circ\Lambda$. We define the space
\[
    \mathcal R_{\Lambda,\Omega}
    =
    \left\{
        c=(c_{j i})\in \mathbb C^s\otimes\mathbb C^r:
        \sum_{j=1}^s\sum_{i=1}^r
        c_{j i}L_j V_i=0
    \right\}.
\]
Let
\[
    \mathcal E_{\Lambda,\Omega}
    =
    \mathcal R_{\Lambda,\Omega}^{\perp}
    \subseteq
    \mathbb C^s\otimes\mathbb C^r,
\]
and let \(P_{\mathcal E}\) denote the orthogonal projection onto
\(\mathcal E_{\Lambda,\Omega}\).

The following formula is the finite-dimensional form of the compression
model above, written in the coefficient space of the composite Kraus
family.
\begin{theorem}\label{th-kraus-formul}
Under the identifications
\[
    \pi_\Omega(\M_m)' \cong \M_r,
    \qquad
    \pi_{\Omega\circ\Lambda}(\M_\ell)'
    \cong B(\mathcal E_{\Lambda,\Omega}),
\]
the pullback transfer map is represented by
\begin{align}\label{q-finit}
    \mathcal T_\Lambda^\Omega(D)
    =
    P_{\mathcal E}
    (I_s\otimes D)
    P_{\mathcal E}
    \big|_{\mathcal E_{\Lambda,\Omega}},
    \qquad D\in \M_r.
\end{align}
\end{theorem}

\begin{proof}
We first identify the abstract domain and codomain of the pullback
transfer map in the present finite-dimensional model.

The minimal Kraus representation of \(\Omega\) gives a minimal
Stinespring representation on  $ \mathbb C^m\otimes \mathbb C^r$ by
\[
    \pi_\Omega(B)=B\otimes I_r,
    \qquad
    \mathsf V_\Omega \xi
    =
    \sum_{i=1}^r V_i\xi\otimes f_i\qquad (B\in\M_m),
\]
where $f_i's$ are   the standard orthonormal basis of  $\mathbb{C}^r$.
Indeed, $\mathsf V_\Omega: \C^n\to\C^m\otimes\C^r$ and
\[
    \mathsf V_\Omega^*(B\otimes I_r)\mathsf V_\Omega
    =
    \sum_{i=1}^r V_i^*BV_i
    =
    \Omega(B).
\]
Therefore
\[
    \pi_\Omega(\M_m)'
    =
    (\M_m\otimes I_r)'
    =
    I_m\otimes \M_r.
\]
Thus, under the identification
\[
    \pi_\Omega(\M_m)'\cong \M_r,
\]
an element \(D=[d_{ij}]\in \M_r\) represents the commutant operator
\(I_m\otimes D\). Hence  we have
\begin{align}\label{kraus-omega}
    \Omega_D(B)
    =
    \mathsf V_\Omega^*
    (B\otimes I_r)
    (I_m\otimes D)
    \mathsf V_\Omega
       =
    \sum_{i,j=1}^r d_{ij}V_i^*BV_j.
\end{align}
Next consider the composite Kraus family
\[
    W_{j i}=L_j V_i.
\]
 The coefficient space of the family \(\{W_{ji}\}\) is
\(\mathbb C^s\otimes\mathbb C^r\).  Its relation space is precisely
\(\mathcal R_{\Lambda,\Omega}\).  Passing to the orthogonal complement
\[
    \mathcal E_{\Lambda,\Omega}
    =
    \mathcal R_{\Lambda,\Omega}^{\perp}
\]
removes exactly the linear redundancies in the family \(\{W_{ji}\}\).
Thus \(\mathcal E_{\Lambda,\Omega}\) is the minimal coefficient space
for the composite Kraus family.

  Hence the minimal Stinespring representation of
\(\Omega\circ\Lambda\) is
\[
    \mathbb C^\ell\otimes\mathcal E_{\Lambda,\Omega}
\]
with
\[
    \rho(A)=A\otimes I_{\mathcal E_{\Lambda,\Omega}}.
\]
Therefore
\[
    \pi_{\Omega\circ\Lambda}(\M_\ell)'
    \cong
    I_\ell\otimes B(\mathcal E_{\Lambda,\Omega})
    \cong
    B(\mathcal E_{\Lambda,\Omega}).
\]
Let \(D=[d_{ij}]\in \M_r\) be positive. It follows from \eqref{kraus-omega} and \eqref{kraus-lambda} that
\begin{align}
    (\Omega_D\circ\Lambda)(A) &=
               \sum_{\mu=1}^s\sum_{i,j=1}^r
    d_{ij}
    (L_\mu V_i)^*A(L_\mu V_j).
\end{align}
Since \(W_{\mu i}=L_\mu V_i\), this can be written as
\[
    (\Omega_D\circ\Lambda)(A)
    =
    \sum_{\mu,\nu=1}^s
    \sum_{i,j=1}^r
    (I_s\otimes D)_{\mu i,\nu j}
    W_{\mu i}^*A W_{\nu j},
\]
in which  $(I_s\otimes D)_{\mu i,\nu j}  = \delta_{\mu\nu}d_{ij}$.
The family \(\{W_{\mu i}\}\) may be linearly dependent.  The minimal Stinespring space of
\(\Omega\circ\Lambda\) is obtained by restricting this coefficient space
to
\[
    \mathcal E_{\Lambda,\Omega}
    =
    \mathcal R_{\Lambda,\Omega}^{\perp}.
\]
Hence the Radon--Nikodym derivative with respect to the minimal
Stinespring representation of \(\Omega\circ\Lambda\) is obtained by
compressing \(I_s\otimes D\) to this subspace:
\[
    \frac{\mathrm{d}(\Omega_D\circ\Lambda)}
         {\mathrm{d}(\Omega\circ\Lambda)}
    =
    P_{\mathcal E}
    (I_s\otimes D)
    P_{\mathcal E}
    \big|_{\mathcal E_{\Lambda,\Omega}}.
\]
By the defining identity of the pullback transfer map,
\[
    \mathcal T_\Lambda^\Omega(D)
    =
    \frac{\mathrm{d}(\Omega_D\circ\Lambda)}
         {\mathrm{d}(\Omega\circ\Lambda)}.
\]
Therefore \eqref{q-finit} holds for every \(D\geq0\). Since every element
of \(\M_r\) is a linear combination of positive matrices, and both sides
of \eqref{q-finit} are linear in \(D\), the formula holds for all
\(D\in\M_r\).
 \end{proof}


Assume \(\Omega\circ\Lambda\neq 0\). Applying the finite-dimensional compression formula of Theorem 4.1 to projection operators immediately yields the following algebraic criterion for the preservation of sharpness.
\begin{corollary}\label{co-proj}
Let \(P\in \M_r\) be a projection, and let \(\Omega_P\leq \Omega\) be the
corresponding sharp submap. Then the following are equivalent:

\begin{enumerate}
    \item the pullback \(\Omega_P\circ\Lambda\leq \Omega\circ\Lambda\)
    is sharp;

    \item \(\mathcal T_\Lambda^\Omega(P)\) is a projection;

    \item \(\mathcal E_{\Lambda,\Omega}\) reduces \(I_s\otimes P\);

    \item
    \(
        P_{\mathcal E}(I_s\otimes P)
        =
        (I_s\otimes P)P_{\mathcal E}.
    \)
\end{enumerate}
\end{corollary}
\begin{proof}
The Radon--Nikodym derivative of
\(\Omega_P\circ\Lambda\) with respect to \(\Omega\circ\Lambda\) is
\(\mathcal T_\Lambda^\Omega(P)\). Hence
\(\Omega_P\circ\Lambda\) is sharp if and only if
\(\mathcal T_\Lambda^\Omega(P)\) is a projection.
It follows from Theorem~\ref{th-kraus-formul} that
\[
    \mathcal T_\Lambda^\Omega(P)
    =
    P_{\mathcal E}(I_s\otimes P)P_{\mathcal E}
    \big|_{\mathcal E_{\Lambda,\Omega}}.
\]
Since \(I_s\otimes P\) is itself a projection, its compression to
\(\mathcal E_{\Lambda,\Omega}\) is a projection if and only if
\(\mathcal E_{\Lambda,\Omega}\) reduces \(I_s\otimes P\). Equivalently,
\[
    P_{\mathcal E}(I_s\otimes P)
    =
    (I_s\otimes P)P_{\mathcal E}.
\]
This proves the equivalences.
\end{proof}

\begin{remark}\label{rem-mixed-relations}
The criterion in Corollary~\ref{co-proj} has a concrete interpretation
in terms of linear relations among the composite Kraus operators
\[
    W_{ji}=L_jV_i .
\]
Indeed, since
\[
    \mathcal E_{\Lambda,\Omega}
    =
    \mathcal R_{\Lambda,\Omega}^{\perp},
\]
and \(I_s\otimes P\) is self-adjoint, the subspace
\(\mathcal E_{\Lambda,\Omega}\) reduces \(I_s\otimes P\) if and only if
the relation space \(\mathcal R_{\Lambda,\Omega}\) reduces
\(I_s\otimes P\).

Thus sharpness of the pullback
\[
    \Omega_P\circ\Lambda\leq \Omega\circ\Lambda
\]
is equivalent to the following splitting property: every relation
\[
    \sum_{j=1}^s\sum_{i=1}^r c_{ji}L_jV_i=0
\]
splits into two independent relations along
\[
    \mathbb C^s\otimes\mathbb C^r
    =
    (\mathbb C^s\otimes P\mathbb C^r)
    \oplus
    (\mathbb C^s\otimes (I-P)\mathbb C^r).
\]

Equivalently, after choosing a basis in which
\[
    P=
    \begin{pmatrix}
        I_k & 0\\
        0 & 0
    \end{pmatrix},
\]
sharpness is preserved if and only if every relation
\[
    \sum_{j=1}^s\sum_{i\leq k} a_{ji}L_jV_i
    +
    \sum_{j=1}^s\sum_{i>k} b_{ji}L_jV_i
    =
    0
\]
implies
\[
    \sum_{j=1}^s\sum_{i\leq k} a_{ji}L_jV_i=0,
    \qquad
    \sum_{j=1}^s\sum_{i>k} b_{ji}L_jV_i=0.
\]
Therefore sharpness is lost precisely when there exists a linear relation
among the composite Kraus operators which genuinely mixes the
\(P\)-sector and the \((I-P)\)-sector.
\end{remark}

The preceding results provide qualitative criteria for the preservation of a sharp submap. We now use the finite-dimensional compression formula to introduce a quantitative measure for the loss of sharpnes.

 \begin{definition}\label{def-sharp-defect}
Let  \(P\in\M_r\) be a
projection.  The \emph{sharpness-loss defect} of \(P\) under pullback by
\(\Lambda\) is the positive operator
\[
    \Delta_{\Lambda,\Omega}(P)
    :=
    \mathcal T_\Lambda^\Omega(P)
    -
    \mathcal T_\Lambda^\Omega(P)^2
    \in B(\mathcal E_{\Lambda,\Omega}).
\]
Thus \(\Delta_{\Lambda,\Omega}(P)\) is the Schwarz defect of the unital
completely positive map \(\mathcal T_\Lambda^\Omega\) at the projection
\(P\).
\end{definition}
\begin{corollary}\label{co-loss-index}
Let \(P\in \M_r\) be a projection and put $ R_P=I_s\otimes P$.
Then
\[
    \Delta_{\Lambda,\Omega}(P)
    =
    P_{\mathcal E}R_P(I-P_{\mathcal E})R_PP_{\mathcal E}
    \big|_{\mathcal E_{\Lambda,\Omega}}.
\]
and
\[
    \|\Delta_{\Lambda,\Omega}(P)\|
    =
    \|(I-P_{\mathcal E})R_PP_{\mathcal E}\|^2.
\]
Moreover, $\Delta_{\Lambda,\Omega}(P)=0$ if and only if the pullback $\Omega_P\circ\Lambda\leq \Omega\circ\Lambda$ is sharp.

In particular, \(\|\Delta_{\Lambda,\Omega}(P)\|\) gives a quantitative
measure of the failure of the sharp submap \(\Omega_P\) to remain sharp
after pullback by \(\Lambda\).
\end{corollary}

\begin{proof}
By Theorem~\ref{th-kraus-formul}, we have $\mathcal T_\Lambda^\Omega(P)
    =
    P_{\mathcal E}R_PP_{\mathcal E}
    \big|_{\mathcal E_{\Lambda,\Omega}}$.
Since \(R_P^2=R_P\), we get
\begin{align*}
    \Delta_{\Lambda,\Omega}(P)
    &=
    P_{\mathcal E}R_PP_{\mathcal E}
    -
    P_{\mathcal E}R_PP_{\mathcal E}R_PP_{\mathcal E}
    \big|_{\mathcal E_{\Lambda,\Omega}}       \\
    &=
    P_{\mathcal E}R_P(I-P_{\mathcal E})R_PP_{\mathcal E}
    \big|_{\mathcal E_{\Lambda,\Omega}}.
\end{align*}
This is the same as
\[
    \big((I-P_{\mathcal E})R_PP_{\mathcal E}\big)^*
    \big((I-P_{\mathcal E})R_PP_{\mathcal E}\big)
    \big|_{\mathcal E_{\Lambda,\Omega}},
\]
and so \[
    \|\Delta_{\Lambda,\Omega}(P)\|
    =
    \|(I-P_{\mathcal E})R_PP_{\mathcal E}\|^2.
\]

Finally, \(\Delta_{\Lambda,\Omega}(P)=0\) if and only if
\[
    (I-P_{\mathcal E})R_PP_{\mathcal E}=0.
\]
Since \(R_P\) and \(P_{\mathcal E}\) are projections, this is equivalent
to \(\mathcal E_{\Lambda,\Omega}\) reducing \(R_P=I_s\otimes P\).  By
Corollary~\ref{co-proj}, this is equivalent to sharpness of
\(\Omega_P\circ\Lambda\).
\end{proof}

\begin{corollary}
Assume $\Omega\circ\Lambda\neq 0$. With the notation above, the following are equivalent:
\begin{enumerate}
    \item \(\mathcal T_\Lambda^\Omega\) is a \(*\)-homomorphism;

    \item \(\mathcal T_\Lambda^\Omega(P)\) is a projection for every
    projection \(P\in\M_r\);

    \item \(\mathcal E_{\Lambda,\Omega}\) reduces \(I_s\otimes \M_r\);

    \item there exists a subspace $\mathcal F\subseteq \mathbb C^s$  such that $\mathcal E_{\Lambda,\Omega}  = \mathcal F\otimes \mathbb C^r$.
\end{enumerate}

Equivalently, pullback preserves all sharp submaps of \(\Omega\) if and
only if the minimal composite Kraus-index space has the tensor form
\[
    \mathcal E_{\Lambda,\Omega}
    =
    \mathcal F\otimes\mathbb C^r .
\]
\end{corollary}

\begin{proof}
 Note that a unital completely positive map is a \(*\)-homomorphism if and only if
its multiplicative domain is the whole domain. On the other hand, applying  Theorem~\ref{th-loss} with
\(\Gamma_\Lambda^\Omega(D)=I_s\otimes D\) and
\(Q_{\Omega,\Lambda}=P_{\mathcal E}\) implies that
\[
    \operatorname{MD}(\mathcal T_\Lambda^\Omega)=\M_r\
\]
if and only if $\mathcal E_{\Lambda,\Omega}$ reduces \(I_s\otimes D\) for every \(D\in\M_r\). This is precisely the
condition that $\mathcal E_{\Lambda,\Omega}$  reduces $I_s\otimes \M_r$.
Hence (1) and (3) are equivalent.

If \(\mathcal T_\Lambda^\Omega\) is a \(*\)-homomorphism, then it sends
projections to projections, so (1) implies (2).

 Suppose that
\(\mathcal T_\Lambda^\Omega(P)\) is a projection for every projection
\(P\in\M_r\). Corollary~\ref{co-proj} then implies that $\mathcal E_{\Lambda,\Omega}$
reduces \(I_s\otimes P\) for every projection \(P\in\M_r\). Since the
projections linearly span \(\M_r\), it follows that $\mathcal E_{\Lambda,\Omega}$
     reduces $I_s\otimes \M_r$. Thus (2) implies (3).

A subspace \(\mathcal E\subseteq \mathbb C^s\otimes\mathbb C^r\)
reduces \(I_s\otimes\M_r\) if and only if its projection
\(P_{\mathcal E}\) commutes with \(I_s\otimes\M_r\). As
\[
    (I_s\otimes\M_r)'=\M_s\otimes I_r,
\]
this is equivalent to $P_{\mathcal E}=P_{\mathcal F}\otimes I_r$
for some projection \(P_{\mathcal F}\in\M_s\). Equivalently,
\[
    \mathcal E_{\Lambda,\Omega}
    =
    \mathcal F\otimes\mathbb C^r,
    \qquad
    \mathcal F=P_{\mathcal F}\mathbb C^s .
\]
This proves the equivalence of (3) and (4), and the proof is complete.
\end{proof}
\begin{remark}[Postcomposition]
There is an analogous order-theoretic question for postcomposition.  If
\(\Omega:\mathscr A\to B(\mathcal H)\) and
\(\Theta:B(\mathcal H)\to B(\mathcal K)\) are completely positive maps,
then \(0\leq\Phi\leq\Omega\) implies
\[
    0\leq \Theta\circ\Phi\leq \Theta\circ\Omega .
\]
Thus we  also may ask when the Radon--Nikodym derivative of
\(\Theta\circ\Phi\) with respect to \(\Theta\circ\Omega\) remains
projection-valued.  The present paper focuses on precomposition, where
the relation-space formula above gives a concrete sharpness criterion.
We do not pursue the postcomposition case here.
\end{remark}


\subsection{A finite classical interpretation}
 We end with the finite commutative case, which serves as a consistency
check and as intuition for the term ``sharpness loss.''  In this case
the pullback formula reduces to ordinary conditional probability.
In this setting, the loss of sharpness has the elementary meaning that
a definite event before observation becomes a posterior probability
after observation.

Let \(X\) and \(Y\) be finite sets. Let \(\mu\) be a probability measure
on \(X\), and let \(K(y\mid x)\) be a stochastic matrix from \(X\) to \(Y\).
Let $ \Omega:C(X)\to \mathbb C$ be the state
\[
    \Omega(f)=\sum_{x\in X}\mu(x)f(x).
\]
Define $\Lambda:C(Y)\to C(X)$ by
\[
    (\Lambda g)(x)=\sum_{y\in Y}K(y\mid x)g(y).
\]
Then
\[
    (\Omega\circ\Lambda)(g)
    =
    \sum_{y\in Y}\nu(y)g(y),
\]
where
\[
    \nu(y)=\sum_{x\in X}\mu(x)K(y\mid x).
\]
We assume that \(\nu(y)>0\).
  The stochastic matrix \(K\) and the probability measure \(\mu\) cab be used to give
a joint probability measure on \(X\times Y\) by
\[
    \mathbb P(X=x,Y=y)=\mu(x)K(y\mid x).
\]
Assume that \(\mathbb P(Y=y)>0\). For a subset \(S\subseteq X\), assume that
\[
    p_S(y):=\mathbb P(X\in S\mid Y=y).
\]
Equivalently,
\[
    p_S(y)
    =
    \frac{
        \sum_{x\in S}\mu(x)K(y\mid x)
    }{
        \sum_{x\in X}\mu(x)K(y\mid x)
    }.
\]

We claim that
\[
    \frac{\mathrm{d}(\Omega_{1_S}\circ\Lambda)}
         {\mathrm{d}(\Omega\circ\Lambda)}
    =
    p_S.
\]
Indeed, for every \(g\in C(Y)\),
\[
    (\Omega_{1_S}\circ\Lambda)(g)
    =
    \sum_{y\in Y}p_S(y)g(y)\,\mathbb P(Y=y)
    =
    (\Omega\circ\Lambda)(p_Sg).
\]
Therefore
\[
    \mathcal T_\Lambda^\Omega(1_S)
    =
    p_S.
\]
Thus the pullback transfer map sends the sharp event \(1_S\) to the
conditional probability
\[
    \mathcal T_\Lambda^\Omega(1_S)(y)
    =
    \mathbb P(X\in S\mid Y=y).
\]
Consequently, the sharp event \(S\) remains sharp after observation if
and only if
\[
    \mathbb P(X\in S\mid Y=y)\in\{0,1\}
\]
for every \(y\) with \(\mathbb P(Y=y)>0\).

\subsection{An Explicit Matrix Example of Sharpness Loss}
The following finite-dimensional example illustrates Corollary~\ref{co-proj} and the defect $\Delta_{\Lambda,\Omega}(P)$. It shows exactly how precomposition can turn a sharp Radon--Nikodym derivative (a projection) into a non-sharp positive contraction.

Let \(\Omega:\M_2\to\mathbb{C}\) be the standard trace map, \(\Omega(X)=\mathrm{Tr}(X)\). Its minimal Kraus representation requires two operators from \(\mathbb{C}\) to \(\mathbb{C}^2\):
\[
    V_1 = \begin{pmatrix} 1 \\ 0 \end{pmatrix}, \qquad
    V_2 = \begin{pmatrix} 0 \\ 1 \end{pmatrix}.
\]
Here, \(\pi_\Omega(\M_2)' \cong \M_2\). A positive contraction \(D \in \M_2\) defines a submap
\[
\Omega_D(X) = \sum_{i,j=1}^2 d_{ij} V_i^* X V_j = \mathrm{Tr}(D^T X).
\]
If \(D\) is a projection matrix, \(\Omega_D\) is sharp relative to \(\Omega\).

Now, define a completely positive map \(\Lambda:\M_2\to\M_2\) parametrized by \(p \in (0,1)\), given by the Kraus operators:
\[
    L_1 = \sqrt{1-p} \begin{pmatrix} 1 & 0 \\ 0 & 1 \end{pmatrix}, \qquad
    L_2 = \sqrt{p} \begin{pmatrix} 1 & 0 \\ 0 & -1 \end{pmatrix}.
\]
The composite Kraus operators \(W_{ji} = L_j V_i\) are:
\[
    W_{11} = \sqrt{1-p}\, e_1, \quad W_{12} = \sqrt{1-p}\, e_2, \quad W_{21} = \sqrt{p}\, e_1, \quad W_{22} = -\sqrt{p}\, e_2,
\]
where \(e_1, e_2\) are the standard basis vectors of \(\mathbb{C}^2\). The linear redundancies in this family form the composite relation space \(\mathcal{R}_{\Lambda,\Omega} \subset \mathbb{C}^2 \otimes \mathbb{C}^2\), which is strictly spanned by the two independent relation vectors:
\[
    r_1 = \sqrt{p}\, e_1 \otimes e_1 - \sqrt{1-p}\, e_2 \otimes e_1, \qquad
    r_2 = \sqrt{p}\, e_1 \otimes e_2 + \sqrt{1-p}\, e_2 \otimes e_2.
\]

We evaluate the sharpness preservation for two distinct projection matrices using the structural criterion.

\textbf{Case 1: A diagonal projection.}
Let \(P = E_{11} = \begin{pmatrix} 1 & 0 \\ 0 & 0 \end{pmatrix}\). We test whether the relation space reduces \(I_2 \otimes P\). Applying the projection gives:
\[
    (I_2 \otimes P) r_1 = r_1 \in \mathcal{R}_{\Lambda,\Omega}, \qquad (I_2 \otimes P) r_2 = 0 \in \mathcal{R}_{\Lambda,\Omega}.
\]
Because \(\mathcal{R}_{\Lambda,\Omega}\) reduces \(I_2 \otimes P\), the defect \(\Delta_{\Lambda,\Omega}(P) = 0\). The submap \(\Omega_P \circ \Lambda\) remains exactly sharp.

\textbf{Case 2: A dense projection.}
Let \(Q = \frac{1}{2}\begin{pmatrix} 1 & 1 \\ 1 & 1 \end{pmatrix}\). Applying \(I_2 \otimes Q\) to \(r_1\) yields:
\[
    (I_2 \otimes Q) r_1 = \frac{1}{2} r_1 + \frac{1}{2}\Big( \sqrt{p}\, e_1 \otimes e_2 - \sqrt{1-p}\, e_2 \otimes e_2 \Big).
\]
Let \(v\) denote the second term. Because \(r_1\) belongs exclusively to the \(\cdot \otimes e_1\) subspace and \(r_2\) belongs exclusively to the \(\cdot \otimes e_2\) subspace, for \(v\) to reside in the span of \(\{r_1, r_2\}\), it must be a scalar multiple of \(r_2\). However, comparing the coefficients of \(v\) and \(r_2\) yields the requirement \(\frac{1}{2} = -\frac{1}{2}\), which is a contradiction. Thus, \((I_2 \otimes Q) r_1 \notin \mathcal{R}_{\Lambda,\Omega}\).

Because the relation space does not reduce $I_2 \otimes Q$, Corollary~\ref{co-proj} implies that sharpness is lost. As in Remark~\ref{rem-mixed-relations}, the mixed Kraus relations of $\Lambda$ mix the $Q$-sector and the $(I-Q)$-sector. Consequently, the projection $Q$ becomes a non-sharp positive contraction under the pullback transfer map, and the magnitude of this deviation is quantified by the defect norm $\|\Delta_{\Lambda,\Omega}(Q)\|$

\medskip

\noindent \textit{Conflict of Interest Statement.}  There is no conflict of interest.
\medskip

\noindent \textit{Ethical Statement.}  Not applicable. This research did not involve human participants, personal data, or animals.

\medskip
\noindent \textit{Informed Consent.} Not applicable.

\medskip
\noindent \textit{Funding Statement.} No funding was received for conducting this study.

\medskip
\noindent\textit{Data Availability Statement.} Data sharing not applicable to this article as no datasets were generated or analyzed during the current study.

  \medskip
\bibliographystyle{amsplain}

\end{document}